\documentclass[english,a4paper]{article}
\usepackage[T1]{fontenc}
\usepackage[utf8]{inputenc}
\usepackage[english]{babel}
\usepackage{amsmath}
\usepackage{amssymb}
\usepackage{verbatim}
\makeatletter

\newtheorem{theorem}{Theorem}

\newtheorem{claim}{Proposition}

\newtheorem{lemma}{Lemma}
\def\@opargbegintheorem#1#2#3{%
  \trivlist
  \item[\MakeLinkTarget{\@currentcounter}\hskip\labelsep
        {\bfseries #1\ #2\ (#3).}]\itshape}

\newenvironment{remark}{\par\noindent\textbf{Remark.}\ }{\par}

\newenvironment{proof}[1][Proof]{%
  \par\noindent
  \if\relax\detokenize{#1}\relax\else\textit{#1.}\ \fi
}{\nobreak\hfill\rule{0.5em}{0.5em}\par}
\makeatother

\title{Distribution of the First Hitting Point on a Boundary Set for a Random Walk in a Lattice Strip}
\author{I.M. Kuzin,  A.A. Zamyatin \\ [0.5em]
         Lomonosov Moscow State University}
\date{}
\begin{document}
\maketitle
\begin{abstract}
We consider a time-homogeneous random walk in a strip of a two-dimensional lattice whose transition probabilities are invariant under shifts along one of the coordinates. We study the limiting distribution of the first hitting point into a boundary set as the initial condition tends to infinity. Explicit expressions are obtained for this limiting distribution for negative and zero mean drift.
\end{abstract}

\section{Introduction}

We consider a time-homogeneous random walk
$\xi_n=(\xi_n^1,\xi_n^2)$ in a strip of the two-dimensional lattice
$\Pi=\mathbb Z\times S$, where $S=\{1,\ldots,N\}$. We assume that its transition probabilities are invariant under translations along the first coordinate, while the first
coordinate changes by at most $1$ per time step; thus, the possible increments in the first coordinate are $-1,0,1$.
Such a walk belongs to the class of random walks in domains with boundaries \cite{FMM}.

An important characteristic of the model is the mean drift in the first coordinate,
\[
v=\sum_{j=1}^N \pi_j m_j,
\]
where $m_j$ is the mean increment of the first coordinate conditional on the second coordinate being $j$, and $\pi_j$ is the stationary distribution of the induced chain \cite{FMM} (defined precisely below).

In this paper we obtain the generating function of the distribution of the first hitting point on the set
$ A=\{(0,j):j\in S\} $; we prove that, as the initial first coordinate $i_0 \to \infty$, the limiting distribution of the first hitting point on this set coincides with the
stationary distribution $\nu$ of the embedded Markov chain
$\eta_k=\xi^2_{\tau_k}$, which records the second component $\xi^2_{\tau_k}$ at the random times $\tau_k$ when the first coordinate successively decreases by $1$.
 
Theorems \ref{th2} and \ref{th3} provide explicit expressions for the stationary distribution $\nu$ in the cases $v<0$ and $v=0$, respectively.

A finite algorithm is given for computing the transition matrix of the embedded
chain.

We also apply the results to a system of independent particles starting
from sites whose first coordinate does not exceed $n$. A strong law of large numbers is proved for the proportions of particles absorbed at each point of $A$,
providing a statistical
interpretation of the stationary distribution of the embedded chain. Finally, we present
an example of a random walk in a strip of height $2$, for which all the principal quantities
are computed explicitly.

The paper is organized as follows. Section \ref{s2}
introduces the main definitions and states Theorems~1--3.
Section \ref{s3}
contains auxiliary Lemmas~2--9 and the proofs of Theorems~1--3, and
describes a finite spectral algorithm for computing the transition matrix of the embedded chain. In
Section \ref{s4},
the results are applied to a system of independent particles moving in the strip; here we state
and prove Theorem~4. In the final
section, all the principal quantities are calculated explicitly for a random walk in a strip of height~2.

\section{Main results} \label{s2}
Consider a time-homogeneous Markov chain $\xi_{n}=(\xi_{n}^{1},\xi_{n}^{2})$
on the state space $\Pi=\mathbb Z\times S$, where $S=\{1,2,\ldots,N\}$. Define its transition
probabilities by
\[
p_{l_{1}j_{1}}^{lj}=P(\xi_{n+1}=(l_{1},j_{1})\;|\;\xi_{n}=(l,j))
\]
where $(l,j),(l_{1},j_{1})\in\Pi$.

We impose the following assumptions:
\begin{itemize}

\item Bounded increments in the first coordinate: if $|l-l_{1}|>1$, then
\[
p_{l_{1}j_{1}}^{lj}=0
\]
\item If $|l-l_1|\leq 1$ and $\forall j,j_1 \in S$, then
\[
p^{lj}_{l_1j_1}>0.
\]
\item Translation invariance in the first coordinate:
\begin{equation}
p_{l_{1}j_{1}}^{lj}=p_{l_{1}-l,j_{1}}^{0j}\label{hom}
\end{equation}
\end{itemize}

Define the induced Markov chain as the finite-state Markov
chain on $S$ with transition probabilities
\begin{equation}
_{0}p_{j_{1}}^{j}=p_{l+1j_{1}}^{lj}+p_{lj_{1}}^{lj}+p_{l-1j_{1}}^{lj}.\label{trans_prob_ind_chain}
\end{equation}
Clearly, this chain is irreducible and aperiodic.

We use the notation
\[
\mathbf 1:=(1,\ldots,1)^T\in\mathbb R^N
\]for the vector whose components are all equal to \(1\).
Let $\pi=(\pi_j)_{j\in S}$ denote the stationary distribution of the induced
chain. Henceforth, $\pi$ is regarded as a column vector, so that $\pi^TP=\pi^T$ and $\pi^T\mathbf1=1$. Define the mean increment in the first coordinate in phase $j$ by
\begin{equation}
 m_j=E\bigl(\xi_{n+1}^1-\xi_n^1\mid \xi_n=(l,j)\bigr),\qquad j\in S.
 \label{mean_jump}
\end{equation}
By translation invariance, this conditional expectation depends
neither on $n$ nor on $l$.

\begin{claim} For the first coordinate of the chain $\xi_n$, we have
\[
\frac{\xi_n^1}{n}\xrightarrow[n\to\infty]{a.s.}\sum_{j=1}^N\pi_jm_j.
\]
\end{claim}
\begin{proof}
Set $\Delta_k=\xi_{k+1}^1-\xi_k^1$ and
$\mathcal F_k=\sigma(\xi_0,\ldots,\xi_k)$. By the law of total
expectation and the definition of $m_j$,
\[
 E(\Delta_k\mid\mathcal F_k)=m_{\xi_k^2}.
\]
Thus $D_{k+1}:=\Delta_k-m_{\xi_k^2}$ form a sequence of bounded
martingale differences. By the strong law of large numbers for martingale
differences (see, e.g., \cite{base5}),
\[
 \frac1n\sum_{k=0}^{n-1}D_{k+1}\xrightarrow[n\to\infty]{a.s.}0.
\]
On the other hand, $\xi_n^2$ is a finite-state irreducible aperiodic Markov chain
with stationary distribution $\pi$; hence, by the ergodic theorem,
\[
 \frac1n\sum_{k=0}^{n-1}m_{\xi_k^2}
 \xrightarrow[n\to\infty]{a.s.}\sum_{j=1}^N\pi_jm_j.
\]
Since
\[
 \xi_n^1-\xi_0^1=\sum_{k=0}^{n-1}\Delta_k
 =\sum_{k=0}^{n-1}m_{\xi_k^2}+\sum_{k=0}^{n-1}D_{k+1},
\]
the assertion follows.
\end{proof}
Let
\begin{equation}
 v=\sum_{j=1}^{N}\pi_jm_j.\label{induced_vector}
\end{equation}
Thus, $v$ is the mean drift in the first coordinate.

\paragraph{The embedded Markov chain.}
Let $i_0\in\mathbb N$ and $\xi_0^1=i_0$. Set $\tau_0=0$ and, for
$k=1,\ldots,i_0$
\[
 \tau_k=\inf\{n\geq0:\xi_n^1=i_0-k\}.
\]
In particular, $\tau_{i_0}$ is the first hitting time of level $0$.
By the strong Markov property and translation invariance, the distribution of the chain's state at the next descent by one level depends on the past only through \(\xi^2_{\tau_k}\). Consequently,
\[
\eta_k=\xi^2_{\tau_k}
\]
is a time-homogeneous Markov chain.

If \(v=0\), the first coordinate of the chain is almost surely unbounded both above and below:
\[
\liminf_{n\to\infty}\xi_n^1=-\infty,\qquad
\limsup_{n\to\infty}\xi_n^1=+\infty
\](see \cite[Example~6.1]{base7}).
Since increments in the first coordinate can only be \(-1,0,1\), every lower level is almost surely reached. Therefore,
\[
P(\tau_k<\infty)=1.
\]
Consequently, every lower level is reached almost surely, and hence
$P(\tau_k<\infty)=1$ for all $k\leq i_0$.

Consider the set $A=\{(0,j):j\in S\}\subset\Pi$ and define
\[
 h_{0j}^{i_0j_0}
 =P\{\xi^2_{\tau_{i_0}}=j,\,\tau_{i_0}<\infty\mid\xi_0=(i_0,j_0)\}.
\]
When $v\leq0$, we have $P(\tau_{i_0}<\infty)=1$, so
\[
 h_{0j}^{i_0j_0}
 =P\{\xi^2_{\tau_{i_0}}=j\mid\xi_0=(i_0,j_0)\}.
\]

Define the embedded chain by
\[
 \eta_0=j_0,\qquad \eta_k=\xi^2_{\tau_k},\quad k=1,\ldots,i_0.
\]
Denote its transition probabilities by
\[
 q_{ij}=P\{\eta_{k+1}=j\mid\eta_k=i\},\qquad i,j\in S,
\]
and set $Q=(q_{ij})_{i,j=1}^N$.

\begin{lemma}\label{l1}
Suppose that $v\leq0$. Then
\[
 h_{0j}^{i_0j_0}
 =\sum_{j_1,\ldots,j_{i_0-1}\in S}
 q_{j_0j_1}q_{j_1j_2}\cdots q_{j_{i_0-1}j}
 =(Q^{i_0})_{j_0j}.
\]
\end{lemma}
\begin{proof}
Set $\hat{q}_{j}^{i}(0)=0$ and
\[
\hat{q}_{j}^{i}(k)=P(\tau_{1}=k,\xi_{\tau_{1}}^{2}=j|\xi_{0}^{2}=i)=\sum p_{l_{1}i_{1}}^{li}p_{l_{2}i_{2}}^{l_{1}i_{1}}...p_{l_{k}i_{k}}^{l_{k-1}i_{k-1}},k>0
\]
where the sum is taken over all paths of length $k$ such that $l_{s}\geq l,s=1,...,k-2$,
$l_{k-1}=l,l_{k}=l-1,i_{k}=j$. By translation invariance (\ref{hom}),
this definition is valid because the right-hand side of the preceding equality
does not depend on $l$.

We may therefore write
\[
q_{j}^{i}=\sum_{k=1}^{\infty}\hat{q}_{j}^{i}(k)
\]

The absorption probability can be expressed as
\[
h_{0j}^{i_{0}j_{0}}=\sum_{n\geq i_{0}}P\{\tau_{i_{0}}=n,\xi_{\tau_{i_{0}}}^{2}=j|\xi_{0}=(i_{0},j_{0})\}=\sum\sum_{n\geq i_{0}}\sum\hat{q}_{j_{1}}^{j_{0}}(k_{1})\hat{q}_{j_{2}}^{j_{1}}(k_{2})...\hat{q}_{j}^{j_{i_{0}-1}}(k_{i_{0}})
\]
The summations run over all $n\geq i_0$, all intermediate phases
$(j_1,\ldots,j_{i_0-1})\in S^{i_0-1}$, and all positive integers
$k_1,\ldots,k_{i_0}$ satisfying $k_1+\cdots+k_{i_0}=n$.

Then
\[
\sum_{n\geq i_{0}}\sum_{k_{1}+k_{2}+...+k_{i_{0}}=n}\hat{q}_{j_{1}}^{j_{0}}(k_{1})\hat{q}_{j_{2}}^{j_{1}}(k_{2})...\hat{q}_{j}^{j_{i_{0}-1}}(k_{i_{0}})=\sum_{k_{1}=1}^{\infty}\sum_{k_{2}=1}^{\infty}\dots\sum_{k_{i_{0}}=1}^{\infty}\hat{q}_{j_{1}}^{j_{0}}(k_{1})\hat{q}_{j_{2}}^{j_{1}}(k_{2})...\hat{q}_{j}^{j_{i_{0}-1}}(k_{i_{0}})=q_{j_{1}}^{j_{0}}q_{j_{2}}^{j_{1}}...q_{j}^{j_{i_{0}-1}}
\]

Consequently,
\[
h_{0j}^{i_{0}j_{0}}=\sum_{j_{1},j_{2},...,j_{i_{0}-1}}q_{j_{1}}^{j_{0}}q_{j_{2}}^{j_{1}}...q_{j}^{j_{i_{0}-1}}=q_{j}^{j_{0}}(i_{0}),
\]
where $q_{j}^{j_{0}}(l)$ is the $l$-step transition probability of the Markov
chain $\eta_{k}$.
\end{proof}
\begin{remark}
If one instead defines $q^i_j=P(\eta_{k+1}=j,\tau_{k+1}<\infty|\eta_{k}=i)$, the assertion of Lemma 1 also holds for $v>0$.
\end{remark}

Clearly, $\eta_n$ is irreducible and aperiodic.
Indeed, for any \(i,j\in S\), if \(\eta_k=i\), then at time \(\tau_k\) the original chain is in some state \((r,i)\), and hence
\[
q_{ij}\ge
\mathbb P\{\xi_{\tau_k+1}=(r-1,j)\mid \xi_{\tau_k}=(r,i)\}
=
p^{ri}_{r-1,j}>0.
\]
Thus all entries of $Q$ are positive. Consequently, the finite-state chain $\eta_k$ is irreducible and aperiodic.

Let $\nu=(\nu_{j},j=1,...,N)$ denote the stationary distribution
of the embedded Markov chain $\eta_{k}$.

\begin{claim}\label{th1}
Suppose that $v\leq0$. As $i_0\to\infty$,
\[
 h_{0j}^{i_0j_0}\longrightarrow\nu_j,\qquad j\in S,
\]
where $\nu=(\nu_j)_{j\in S}$ is the stationary distribution of the embedded chain.
\end{claim}

\begin{proof}
By Lemma~1,
\[
 h_{0j}^{i_0j_0}=(Q^{i_0})_{j_0j}.
\]
Since the finite-state chain with transition matrix $Q$ is irreducible and aperiodic,
$Q^n\to\mathbf 1\nu^T$ as $n\to\infty$. This proves the assertion.
\end{proof}
Set
\[
 \tau_A=\inf\{n\geq0:\xi_n\in A\}.
\]
For $i\in\mathbb Z_+$ and $j,l\in S$, introduce
\[
 h_{0l}^{ij}(n)
 :=P\{\tau_A=n,\,\xi_{\tau_A}=(0,l)\mid\xi_0=(i,j)\},
\]
that is, the probability that the first hit on $A$ occurs at time $n$ and
that the hitting point is $(0,l)$.

Consider the transition probabilities
\[
 p_{i,l}^{i+s,j}=P\bigl((i+s,j)\to(i,l)\bigr),\qquad s=-1,0,1.
\]
By translation invariance, these do not depend on $i$, so we may set
$p_{s,l}^j=p_{i,l}^{i+s,j}$ and
\[
 P_s=(p_{s,l}^j)_{j,l\in S},\qquad s=-1,0,1,
\]
and put $P_s=0$ for $|s|>1$. Then
\[
 \sum_{s\in\mathbb Z}P_s=P,
\]
where $P$ is the transition matrix of the induced chain, and
\[
 P(x)=xP_1+P_0+x^{-1}P_{-1},\qquad x\in\mathbb C\setminus\{0\}.
\]

Introduce the generating functions
\[
 h_{0l}^{ij}(z)=\sum_{n=0}^{\infty}h_{0l}^{ij}(n)z^n
\]
and the vectors
\[
 h_{0l}^{i}(z)=\bigl(h_{0l}^{ij}(z)\bigr)_{j\in S},\qquad i\in\mathbb Z_+.
\]
We have $h_{0l}^{0}(z)=e_l$. At $z=1$, the vector $h_{0l}^{i}:=h_{0l}^{i}(1)$
contains the probabilities of first hitting $(0,l)$ from the states $(i,j)$, $j\in S$.
Let $H$ be the matrix whose $l$-th column is $h_{0l}^{1}$.
By Lemma~1, the matrices $H$ and $Q$ coincide when $v\leq0$.

Consider the generating function of the sequence $\{h_{0l}^{i}\}_{i\geq1}$:
\[
 H_{0l}(x)=\sum_{i=1}^{\infty}x^{i-1}h_{0l}^{i}.
\]

\begin{theorem}\label{the1}
For every $l\in S$,
\begin{gather}
\label{H}
 H_{0l}(x)=(xI-xP(x))^{-1}\bigl(xP_1e_l-P_{-1}h_{0l}^{1}\bigr),
\end{gather}
at all points where the matrix $xI-xP(x)$ is invertible. At points of the unit disk
where the right-hand side has a removable singularity, the equality is understood by analytic
continuation. The vector $h_{0l}^{1}$ can be computed by the finite spectral algorithm
in Section \ref{subsec:label}.
\end{theorem}

The stationary probabilities can be expressed as follows:

\begin{theorem} \label{th2}
If $v<0$, then
\begin{equation}
\label{theor2}
\nu_l=-\frac{\pi^{T} (P_1e_l-P_{-1}h^1_{0l})}{v}, \quad \forall l \in \{1,...N\},
\end{equation}
where $\pi$ is the stationary distribution vector of the induced Markov chain and $e_l$ is the vector whose
$l$-th component is $1$ and all remaining components are $0$.
\end{theorem}

It turns out that an explicit formula for $\nu$ can also be given when $v=0$.

\begin{theorem} \label{th3} 
If $v=0$, then
\begin{equation}\label{theor3}
\nu_l=\frac{2}{a}
\Bigl[
\pi^T P_1e_l
+b^T\bigl(P_1e_l-P_{-1}h^1_{0l}\bigr)
\Bigr],\qquad l\in S,
\end{equation}
where $\mathbf1=(1,\ldots,1)^T$,
\[
Z=(I-P+\mathbf1\circ\pi)^{-1}-\mathbf1\circ\pi,
\qquad
b^T=\pi^T(P_1-P_{-1})Z,
\]
and the constant $a$ is given by
\[
a=2\pi^TP_{-1}\mathbf1
  +2\pi^T(P_1-P_{-1})Z(P_1-P_{-1})\mathbf1.
\]
Here $\mathbf1\circ\pi=\mathbf1\pi^T$; the identities for $a$ and $b^T$
are derived below in the proof of the theorem.

\end{theorem}

\begin{remark}
Theorems 2 and 3 provide explicit expressions for the stationary distribution of the embedded chain in terms of the transition matrices of the induced chain and the mean drift in the first coordinate.
\end{remark}

\section{Proofs of the main results} \label{s3}
Before proving the main results, we establish several auxiliary lemmas.
For $x\neq0$, define
\[
D(x)=\det(I-P(x)).
\]
Since \(P(x)\) contains \(x^{-1}\), the function \(D(x)\) is initially defined only for \(x\neq0\). Therefore we introduce separately the matrix polynomial
\begin{equation}
\label{xi-px}
F(x):=-x^{2}P_{1}+x(I-P_{0})-P_{-1},\qquad x\in\mathbb C.
\end{equation}
For $x\neq0$, we have
\[
F(x)=x(I-P(x))=xI-xP(x),
\qquad
D(x)=\frac{\det F(x)}{x^N}.
\]
Thus $F$ is also defined at $x=0$, and $\det F(x)$ is a polynomial
of degree at most $2N$. Note also that $D(1)=0$.

Let $H$ be the matrix whose $l$-th column is $h_{0l}^{1}$, $l\in S$.
By Lemma~1, its powers describe successive descents through the levels: the $l$-th column
of $H^n$ is $h_{0l}^{n}$. Conditioning on the first step from $(1,j)$ gives
\[
 h_{0l}^{1j}=p_{1,l}^{j}
 +\sum_{k=1}^{N}p_{0,k}^{j}h_{0l}^{1k}
 +\sum_{k=1}^{N}p_{-1,k}^{j}h_{0l}^{2k}.
\]
Consequently,
\begin{equation}
\label{h_eq}
 H=P_{-1}H^2+P_0H+P_1.
\end{equation}

\begin{lemma}
For every nonnegative solution $X$ of equation \eqref{h_eq},
we have $H\leq X$ entrywise. In particular, $H$ is the minimal
nonnegative solution. If $H$ is stochastic, it is the
unique stochastic solution of equation \eqref{h_eq}.
\end{lemma}
\begin{proof}
Let $X$ be any nonnegative solution of \eqref{h_eq}. For $i\geq0$ and
$l\in S$, set
\[
 u_l(i,j)=(X^i)_{jl},\qquad j\in S,
\]
where $X^0=I$, so that $u_l(0,j)=\delta_{jl}$. Multiplying \eqref{h_eq} on the right by
$X^{i-1}$, $i\geq1$, gives
\[
 X^i=P_{-1}X^{i+1}+P_0X^i+P_1X^{i-1}.
\]
Therefore, for every $i\geq1$ and $j\in S$,
\[
\begin{aligned}
 u_l(i,j)=&\sum_{k=1}^{N}p_{1,k}^{j}u_l(i-1,k)
 +\sum_{k=1}^{N}p_{0,k}^{j}u_l(i,k)+\sum_{k=1}^{N}p_{-1,k}^{j}u_l(i+1,k).
\end{aligned}
\]
It is now easy to see that, for \(i\ge1\),
\[
u_l(i,j)
=
E_{(i,j)}\bigl[u_l(\xi_1^1,\xi_1^2)\bigr],
\]while at level \(0\),
\[
u_l(0,j)=\delta_{jl}.
\]

Let $P_{(i,j)}$ denote the law of the chain starting from $\xi_0=(i,j)$, and let
$\tau=\inf\{n\geq0:\xi_n\in A\}$. For each fixed $n$, define
\[
 M_m=u_l\bigl(\xi^1_{m\wedge\tau},\xi^2_{m\wedge\tau}\bigr),
 \qquad 0\leq m\leq n.
\]
Then $(M_m)_{m=0}^{n}$ is a martingale. Indeed, for $0\leq m\leq n$, when starting from $(i,j)$,
the first coordinate of the stopped chain can take values only in
the finite set $\{0,1,\ldots,i+n\}$, since it changes by
at most one at each step. Hence $M_m$ is integrable:
the function $u_l$ is bounded on this finite set of states.
Moreover, $M_m$ is measurable with respect to $\mathcal F_m$.
For $m<n$, on the event $\{\tau\leq m\}$, we have $M_{m+1}=M_m$.
On the event $\{\tau>m\}$, the Markov property and the identity
$u_l(i,j)=E_{(i,j)}[u_l(\xi_1)]$ for $i\geq1$ yield
\[
 E_{(i,j)}[M_{m+1}\mid\mathcal F_m]
 =u_l(\xi_m^1,\xi_m^2)=M_m.
\]
Consequently, $E_{(i,j)}[M_{m+1}\mid\mathcal F_m]=M_m$
for $0\leq m<n$.
The martingale property implies that
\[
 u_l(i,j)=E_{(i,j)}M_n
 \geq P_{(i,j)}\{\tau\leq n,\,\xi_\tau=(0,l)\}.
\]
Taking the limit as $n\to\infty$ and using continuity of probability from below, we obtain
\[
 u_l(i,j)\geq P_{(i,j)}\{\tau<\infty,\,\xi_\tau=(0,l)\}=h_{0l}^{ij}.
\]
At $i=1$, this means $X\geq H$ entrywise. Thus $H$ is the minimal
nonnegative solution. If both $H$ and $X$ are stochastic, then $H\leq X$ and the sums
of the entries in each of their rows equal one, implying $H=X$. The lemma is proved.
\end{proof}
Consider the principal matrix polynomial
\[
G(x) = P_{-1}x^2 + (P_0-I)x +P_1, \quad x \in \mathbb{C}. 
\]
The polynomial $F$ introduced above is an auxiliary polynomial; for $x\neq0$
it is related to $G$ by $F(x)=-x^2G(x^{-1})$.

Note that $G(H)=0$. Moreover, for $x\neq0$,
\[
G(x)=x^2\bigl(x^{-1}P(x^{-1})-x^{-1}I\bigr)
\]
and therefore
\[
\det G(x)=x^{2N}\det\bigl(x^{-1}P(x^{-1})-x^{-1}I\bigr),\qquad x\neq0.
\]
Equivalently, for $\xi\neq0$,
\begin{equation}
\label{eq1}
\xi^{2N}\det G(\xi^{-1})=\det\bigl(\xi P(\xi)-\xi I\bigr).
\end{equation}

\begin{lemma}
Every eigenvalue $\lambda$ of $H$ is a root of
$\det G(x)=0$. All $N$ eigenvalues of $H$,
counted with algebraic multiplicity, satisfy $|\lambda|\leq1$.
\end{lemma}

\begin{proof} The second assertion is immediate, while the first follows from the generalized B\'ezout theorem (see Chapter~4, Section~2 of \cite{base1}). Indeed, by the generalized B\'ezout theorem, $G(\lambda)$ has the form
$$
G(\lambda)=M(\lambda)(\lambda I-H),
$$
where $M(\lambda)$ is a matrix polynomial. Consequently,
$$
det(\lambda I-H)=0 \Rightarrow detG(\lambda)=0,
$$
which proves the assertion.
\end{proof}

\begin{lemma}
All roots below are counted with algebraic
multiplicity. Then
\begin{enumerate}
    \item \( v < 0 \iff \) the equation \(\det G(x) = 0\) has \(N-1\) roots in the open disk \(|x|<1\) and a simple root at \(x=1\); equivalently, exactly \(N\) roots lie in the closed disk \(|x|\leq1\).
    \item \( v > 0 \iff \) the equation \(\det G(x) = 0\) has \(N\) roots in the open disk \(|x|<1\), while \(x=1\) is a simple root on the boundary.
    \item $v=0 \iff$ the equation $\det G(x) = 0$ has $N-1$ roots in the open disk $|x|<1$ and a double root at $x=1$.
\end{enumerate}
In all three cases, there are no other roots on the unit circle $|x|=1$.
\end{lemma}

\begin{proof}
Apply Theorem~4 of Gail, Hantler, and Taylor~\cite{base6} to
the matrix polynomial
\[
A(x)=xI-(P_1+xP_0+x^2P_{-1})=-G(x).
\]
In the notation of that theorem, $MN=N$, the periodicity parameter is $g=1$,
and the parameter $\gamma$ has the opposite sign to our mean drift $v$.
Thus the assertions concerning the number and multiplicities of the roots of $\det A(x)$
for $v<0$, $v>0$, and $v=0$ follow directly from that theorem.
Since $\det A(x)=(-1)^N\det G(x)$, they also hold for $\det G(x)$. This proves the lemma.
\end{proof}

\begin{lemma}
Suppose $v\leq0$. If $Hu=\lambda u$, $u\neq0$, then
$G(\lambda)u=0$. Conversely, if $|\lambda|<1$ and $G(\lambda)u=0$, then
$Hu=\lambda u$. For $\lambda=1$, we have
$\ker G(1)=\ker(I-H)=\operatorname{span}\{\mathbf 1\}$.
\end{lemma}
\begin{proof}
By the generalized B\'ezout theorem (see Chapter~4, Section~2 of \cite{base1}), there exists
a matrix polynomial $M(x)$ such that
\[
 G(x)=M(x)(xI-H).
\]
Therefore $Hu=\lambda u$ immediately implies $G(\lambda)u=0$.

Consider the converse. From
\[
 \det G(x)=\det M(x)\det(xI-H)
\]
and the root counts in Lemmas~3--4, it follows that $\det M(\lambda)\neq0$ for all
roots of $\det G$ satisfying $|\lambda|<1$: all roots of $\det G$ in the unit disk,
counted with multiplicity, are already accounted for by the eigenvalues of $H$. Thus
$G(\lambda)u=0$ with $|\lambda|<1$ implies $(\lambda I-H)u=0$.

If $v\leq0$, then $H$ is a positive stochastic matrix; hence
its eigenvalue $1$ is simple and $\ker(I-H)=\operatorname{span}\{\mathbf1\}$.
On the other hand, $G(1)=P-I$, and for the positive stochastic matrix $P$,
we likewise have $\ker(P-I)=\operatorname{span}\{\mathbf1\}$. This proves the lemma.
\end{proof}
\begin{lemma}[see {\cite[Chapter~1, Sections~1.4--1.5, Propositions~1.9 and~1.11]{GLR}}]
Suppose that $|\lambda|<1$ and the assumptions of Lemma~5 hold. Nonzero vectors
$u_0,u_1,\ldots,u_{p-1}$ form a right Jordan chain of the matrix $H$,
that is,
\[
 (H-\lambda I)u_0=0,\qquad
 (H-\lambda I)u_j=u_{j-1},\quad j=1,\ldots,p-1,
\]
if and only if, for $j=0,\ldots,p-1$
\begin{equation}
\label{jordan-chain-G}
 \sum_{r=0}^{j}\frac{G^{(r)}(\lambda)}{r!}\,u_{j-r}=0.
\end{equation}
\end{lemma}

\subsection{Proof of Theorem \ref{the1}} \label{ss3.1}
\begin{proof}[]
Let $l\in S$. By the Markov property, for $i\geq1$ and $n\geq1$,
\[
\begin{aligned}
 h_{0l}^{ij}(n)=&\sum_{r\in S}p_{1,r}^{j}h_{0l}^{i-1,r}(n-1)
 +\sum_{r\in S}p_{0,r}^{j}h_{0l}^{i,r}(n-1)\\
 &+\sum_{r\in S}p_{-1,r}^{j}h_{0l}^{i+1,r}(n-1),
\end{aligned}
\]
where
\[
 h_{0l}^{ij}(0)=\mathbf 1_{\{i=0,\,j=l\}}.
\]
Passing to the generating functions, we obtain, for $i\geq1$,
\begin{equation}
\label{strip}
 h_{0l}^{i}(z)
 =z\bigl(P_1h_{0l}^{i-1}(z)+P_0h_{0l}^{i}(z)+P_{-1}h_{0l}^{i+1}(z)\bigr).
\end{equation}
Since $h_{0l}^{0}(z)=e_l$, introduce
\[
 H(x,z)=\sum_{i=1}^{\infty}x^ih_{0l}^{i}(z).
\]
Then
\[
 \sum_{i=1}^{\infty}x^ih_{0l}^{i-1}(z)=x e_l+xH(x,z),
\]
and also
\[
 \sum_{i=1}^{\infty}x^ih_{0l}^{i+1}(z)=x^{-1}H(x,z)-h_{0l}^{1}(z).
\]
Consequently, by \eqref{strip},
\[
 H(x,z)=z\Bigl(P(x)H(x,z)+xP_1e_l-P_{-1}h_{0l}^{1}(z)\Bigr).
\]
Set
\[
 H_0(x,z)=\sum_{i=1}^{\infty}x^{i-1}h_{0l}^{i}(z).
\]
For $x\neq0$, we have $H_0(x,z)=x^{-1}H(x,z)$; at $x=0$, this equality
is understood by analytic continuation, with $H_0(0,z)=h_{0l}^{1}(z)$.
Thus, for $x\neq0$,
\[
 (xI-zxP(x))H_0(x,z)
 =z\bigl(xP_1e_l-P_{-1}h_{0l}^{1}(z)\bigr),
\]
and consequently,
\begin{equation}
\label{strip_gf}
 H_0(x,z)
 =z(xI-zxP(x))^{-1}\bigl(xP_1e_l-P_{-1}h_{0l}^{1}(z)\bigr)
\end{equation}
at all points where the inverse matrix exists. Setting $z=1$, we obtain
\[
 H_{0l}(x)=(xI-xP(x))^{-1}\bigl(xP_1e_l-P_{-1}h_{0l}^{1}\bigr),
\]
which proves formula \eqref{H}. Since the left-hand side is analytic for $|x|<1$,
any singularities of the right-hand side inside the unit disk are removable, and
the formula extends to them analytically.
\end{proof}
\subsection{Algorithm for computing the matrix $H$}\label{subsec:label}
We consider the case $v\leq0$, to which Lemma~5 applies.
All basis vectors are regarded as columns. The vector $\mathbf1$ denotes
the column of $N$ ones (see the definition in Section~2).

First find all zeros of $\det G(x)$ in the open unit disk, including
a possible zero root $\lambda=0$. By Lemmas~3--5, counted with
algebraic multiplicities, these coincide with the eigenvalues of $H$ that lie in
$|x|<1$. If $v\leq0$, we add the simple eigenvalue
$\lambda=1$ with eigenvector $\mathbf1$.

Now fix a root $\lambda$ with $|\lambda|<1$. Following the notation
for root multiplicities used in \cite{base},
let $m(\lambda)$ be the multiplicity of $\lambda$ as a zero of the polynomial
\[
 L(z):=\det G(z).
\]
In other words, $m(\lambda)$ is the unique positive integer for which
\[
 L^{(r)}(\lambda)=0,\qquad r=0,1,\ldots,m(\lambda)-1,
 \qquad L^{(m(\lambda))}(\lambda)\neq0.
\]
In our setting, Lemma~3, the factorization of the
matrix polynomial, and the root count in Lemma~4 imply that for
$|\lambda|<1$, the number $m(\lambda)$ is also the algebraic
multiplicity of $\lambda$ as an eigenvalue of $H$. The boundary
eigenvalue $\lambda=1$ for $v\leq0$ is treated separately, as noted above.
Set
\[
 G_r(\lambda)=\frac{G^{(r)}(\lambda)}{r!},\qquad r\geq0,
\]
and for $k\geq1$ introduce the block matrix
\[
 \mathcal T_k(\lambda)=
 \begin{pmatrix}
 G_0&0&\cdots&0\\
 G_1&G_0&\cdots&0\\
 \vdots&\vdots&\ddots&\vdots\\
 G_{k-1}&G_{k-2}&\cdots&G_0
 \end{pmatrix}.
\]
The vector $(u_0^T,\ldots,u_{k-1}^T)^T$ belongs to
$\ker\mathcal T_k(\lambda)$ if and only if
$u_0,\ldots,u_{k-1}$ satisfy system \eqref{jordan-chain-G}. If
$u_0\neq0$, this sequence is a right root chain of the
matrix polynomial $G$ and, by Lemma~6, corresponds to a Jordan chain
of $H$. Zero elements of the kernel and solutions with $u_0=0$ are retained in the linear
spaces below but do not, by themselves, count as chains of length $k$.

To select the chains explicitly, set
\[
 K_k(\lambda)=\ker\mathcal T_k(\lambda),\qquad
 E_k(\lambda)=\Bigl\{u_0:\ \exists u_1,\ldots,u_{k-1}\text{ such that }
 (u_0^T,\ldots,u_{k-1}^T)^T\in K_k(\lambda)\Bigr\}.
\]
Then
\[
 E_1(\lambda)\supseteq E_2(\lambda)\supseteq\cdots,
\]
and $\dim E_k(\lambda)$ is the number of Jordan blocks for $\lambda$ of size at
least $k$. In particular, the number of blocks of size exactly $k$ is
\[
 \dim E_k(\lambda)-\dim E_{k+1}(\lambda),
\]
and
\[
 m(\lambda)=\sum_{k\geq1}\dim E_k(\lambda).
\]
All spaces $K_k$ and $E_k$ are obtained by solving finite systems of linear
equations. For each $k$, choose a complementary subspace $B_k$ to
$E_{k+1}$ in $E_k$. For each basis vector $u_0$ of $B_k$, solve system
\eqref{jordan-chain-G} for $u_1,\ldots,u_{k-1}$; the resulting sequence
$u_0,\ldots,u_{k-1}$ is a Jordan chain of length $k$. The procedure terminates,
since the sum of the lengths of the constructed chains equals $m(\lambda)$.

After doing this for all $|\lambda|<1$ and adding, for the simple
eigenvalue $\lambda=1$, a Jordan chain of length~$1$ consisting of
the single vector $u_0=\mathbf1$ (there are no subsequent vectors
$u_1,u_2,\ldots$ in this chain), we obtain $N$
linearly independent generalized eigenvectors. Let $S$ be the matrix
formed by these vectors, ordered
by their Jordan chains, and let $J$ be the corresponding block-diagonal
Jordan matrix. Then
\[
 H=SJS^{-1}.
\]
Thus, all columns $h_{0l}^{1}$ can be obtained by a finite algebraic
procedure involving roots of $\det G$, derivatives of $G$, and solutions of finite systems of
linear equations.

\subsection{Proof of Theorem \ref{th2}} \label{ss3.3}
\begin{lemma}
Suppose that for $x\neq0$,
$F(x)=-x^2G(x^{-1})=xI-xP(x)$. If $v<0$, the matrix function $F^{-1}(x)$
has a pole of order one at $x=1$; if $v=0$, it has a pole of order two
there.
\end{lemma}
\begin{proof}
Set
\[
 \chi(x,\mu)=\det(\mu I-P(x)).
\]
Since $P(x)$ is analytic near $x=1$ and $1$ is a simple
eigenvalue of $P(1)=P$, the implicit function theorem gives an
analytic eigenvalue branch $\lambda(x)$ such that
$\lambda(1)=1$ and
\[
 \chi(x,\lambda(x))=0.
\]
In a neighborhood of $(1,1)$, we have
\[
 \chi(x,\mu)=(\mu-\lambda(x))g(x,\mu),
\]
where $g(1,1)\neq0$. Therefore,
\[
 \det(I-P(x))=(1-\lambda(x))\widetilde g(x),
 \qquad \widetilde g(1)\neq0.
\]
Since $F(x)=x(I-P(x))$,
\[
 \det F(x)=x^N(1-\lambda(x))\widetilde g(x).
\]

Choose analytic right and left eigenvectors $r(x)$ and $l(x)$,
normalized so that
\[
 r(1)=\mathbf1,\qquad l(1)=\pi,\qquad l(x)^Tr(x)=1.
\]
Differentiating $P(x)r(x)=\lambda(x)r(x)$ at $x=1$ and multiplying on the left by
$\pi^T$, we obtain
\[
 \lambda'(1)=\pi^TP'(1)\mathbf1
 =\pi^T(P_1-P_{-1})\mathbf1.
\]
By the definition of the mean increment,
\[
 m=(P_{-1}-P_1)\mathbf1,
 \qquad
 v=\pi^T(P_{-1}-P_1)\mathbf1,
\]
and hence
\begin{equation}
\label{lambda-prime-drift}
 \lambda'(1)=-v.
\end{equation}
If $v<0$, then $\lambda'(1)\neq0$, so $\det F(x)$ has a simple zero at
$x=1$. Since $\operatorname{rank}F(1)=N-1$, the matrix
$\operatorname{adj}F(1)$ is nonzero, and the identity
\[
 F^{-1}(x)=\frac{\operatorname{adj}F(x)}{\det F(x)}
\]
implies that $F^{-1}(x)$ has a simple pole.

If $v=0$, then by Lemma~4 the function $\det G(x)$ has a zero of order exactly
order two at $x=1$. For $x\neq0$, the definition of $F$ gives
\[
F(x)=-x^2G(x^{-1}),
\]
hence,
\[
\det F(x)=(-1)^N x^{2N}\det G(x^{-1}).
\]
Thus $\det F(x)$ also has a zero of order exactly two at $x=1$. On the other
hand,
\[
\det F(x)=x^N(1-\lambda(x))\widetilde g(x),\qquad \widetilde g(1)\neq0.
\]
Since $\lambda'(1)=0$ is already known when $v=0$, we conclude that
$\lambda''(1)\neq0$. Consequently, $1-\lambda(x)$ has a zero of order two at
$x=1$; using $\operatorname{adj}F(1)\neq0$, we obtain a pole of order two
for $F^{-1}(x)$. This proves the lemma.
\end{proof}
We shall use the following operation. Let $v$ and $u$ be column vectors of the same dimension. By $v\circ u$ we denote the matrix $M$ with entries $(M)_{ij}=v_iu_j$.

\paragraph{Proof of Theorem \ref{th2}.}
\begin{proof}[]
\[
 H_0(x)=F^{-1}(x)\bigl(xP_1e_l-P_{-1}h_{0l}^{1}\bigr).
\]
By Lemma~7, in a punctured neighborhood of $x=1$, the matrix $F^{-1}(x)$ admits
the Laurent expansion
\begin{gather}
\label{11}
 F^{-1}(x)=\frac{A}{x-1}+C(x),
\end{gather}
where $C(x)$ is analytic near $1$. Therefore,
\begin{gather}
\label{12}
 H_0(x)=\frac{A\bigl(xP_1e_l-P_{-1}h_{0l}^{1}\bigr)}{x-1}
 +C(x)\bigl(xP_1e_l-P_{-1}h_{0l}^{1}\bigr).
\end{gather}

By Proposition~2, we already know that
\[
 h_{0l}^{n}\longrightarrow \nu_l\mathbf1.
\]
Consequently, Abel's theorem for power series yields
\begin{equation}
\label{abel-limit}
 \lim_{x\uparrow1}(1-x)H_0(x)=\nu_l\mathbf1.
\end{equation}
Since $C(x)$ is bounded in some real neighborhood of $1$,
the second term of \eqref{12}, multiplied by $(1-x)$, tends to zero.
Thus the limit in \eqref{abel-limit} is determined entirely by the principal part
of the Laurent expansion. It remains to find the matrix $A$.

Set $\varepsilon=1-x$. Then
\[
F(x)=-(x^{2}P_{1}-x(I-P_{0})+P_{-1})
=F_2\varepsilon^2+F_1\varepsilon+F_0,
\]
where
\begin{equation}\label{F-eps-coeffs}
F_2=-P_1,\qquad F_1=P_0+2P_1-I,\qquad F_0=I-P.
\end{equation}
Since $x-1=-\varepsilon$, equation \eqref{11} gives
\[
F^{-1}(x)=-\frac{A}{\varepsilon}+C_0+O(\varepsilon).
\]
Substituting this into $F^{-1}(x)F(x)=I$ and comparing the coefficients of
$\varepsilon^{-1}$ and $\varepsilon^0$, we obtain
\[
AF_0=0,
\qquad
-AF_1+C_0F_0=I.
\]
Since $F_0=I-P$, the rows of $A$ are left eigenvectors of $P$
corresponding to eigenvalue $1$. Hence,
\[
A=\alpha\pi^T,
\qquad
\alpha=(\alpha_1,\ldots,\alpha_N)^T.
\]
Multiplying $-AF_1+C_0F_0=I$ on the right by $\mathbf1$ and using
$F_0\mathbf1=0$, we have
\[
-AF_1\mathbf1=\mathbf1.
\]
For any $i\in S$, this gives
\[
-\alpha_i\pi^TF_1\mathbf1=1.
\]
But
\[
F_1\mathbf1=(P_1-P_{-1})\mathbf1,
\qquad
\pi^TF_1\mathbf1=-v,
\]
so $\alpha_i=1/v$ for all $i$. Thus,
\begin{equation}\label{F-residue-v}
F^{-1}(x)
=\frac{\mathbf1\pi^T}{v(x-1)}+C(x)
=-\frac{\mathbf1\pi^T}{v(1-x)}+C(x).
\end{equation}

Set
\[
v_l(x)=xP_1e_l-P_{-1}h^1_{0l}.
\]
Then
\[
H_0(x)=-\frac{\mathbf1\pi^Tv_l(x)}{v(1-x)}+C(x)v_l(x).
\]
Since $C(x)$ is analytic near $1$, the function $C(x)v_l(x)$ is bounded there.
Thus, after multiplication by $(1-x)$, its contribution tends to zero. On the other hand,
by Proposition~2 and Abel's theorem, we already know that
\[
\lim_{x\uparrow1}(1-x)H_0(x)=\nu_l\mathbf1.
\]
Consequently,
\[
\nu_l\mathbf1
=-\frac{\mathbf1\pi^T(P_1e_l-P_{-1}h^1_{0l})}{v},
\]
and hence
\[
\nu_l=-\frac{\pi^T(P_1e_l-P_{-1}h^1_{0l})}{v},\qquad l\in S.
\]
The theorem is proved.
\end{proof}
\subsection{Proof of Theorem \ref{th3}}\label{ss3.4}
\begin{lemma}[see {\cite[\S~4.3]{KemenySnell1976}}] Let $Z=(I-P+\mathbf{1}\circ \pi)^{-1}-\mathbf{1}\circ \pi$. Then
\begin{equation}
(I-P)Z=Z(I-P)=I-\mathbf{1}\circ\pi.
\end{equation}
\end{lemma}
Let
\[
\Pi(x)=r(x)\circ\,\ell(x),
\]
where $r(x)$ and $\ell(x)$ are analytic right and left eigenvectors
corresponding to the eigenvalue $\lambda(x)$ of $P(x)$,
normalized by $\ell(x)^T r(x)=1$.

\begin{lemma}[see {\cite[Chapter~I, \S~5.3; Chapter~II, \S~1.4]{Kato1995}}] In a punctured neighborhood of $x=1$, the following representation holds:
\begin{gather}
\label{al_nacer}
(I-P(x))^{-1}=\frac{\Pi(x)}{1-\lambda(x)}+R(x), 
\end{gather}
where $R(x)$ is analytic in a neighborhood of $x=1$.
\end{lemma}

\begin{proof} For brevity, we suppress the arguments of functions throughout this proof. Choose the vectors $\ell$ and $r$ so that $\ell^T r=1$. Then
\[
\Pi^2=r\,\circ \ell\,r\,\circ \ell=r\circ \ell=\Pi,
\]
that is, $\Pi$ is a projection.

Moreover,
\[
P\Pi=P(r\circ \ell)=(\lambda r)\circ \ell=\lambda\Pi,
\]
and similarly $\Pi P=\lambda\Pi$.
Thus,
\begin{equation}\label{eq:comm-rel}
P\Pi=\Pi P=\lambda\Pi.
\end{equation}
Set $Q=I-\Pi$. Then
\begin{equation}\label{eq:PQ-QP}
PQ=QP=P-\lambda\Pi.
\end{equation}

Clearly,
\[
P=(\Pi+Q)P(\Pi+Q).
\]
Expanding the product gives
\[
P=\Pi P\Pi+\Pi P Q+Q P\Pi+Q P Q.
\]
By \eqref{eq:comm-rel} and \eqref{eq:PQ-QP},
\[
\Pi P\Pi=\lambda\Pi,
\qquad
\Pi P Q=\Pi Q P=0,
\qquad
Q P\Pi=Q\lambda\Pi=0.
\]
Consequently,
\begin{equation}\label{eq:P-splitting}
P=\lambda\Pi+Q P Q.
\end{equation}

It follows that
\begin{equation}\label{eq:I-minus-P-splitting}
I-P=(\Pi+Q)-\lambda\Pi-QPQ=(1-\lambda)\Pi+Q(I-P)Q.
\end{equation}

\medskip

Since $1$ is a simple eigenvalue of $P(1)$, all other eigenvalues of $P(1)$ differ from $1$. By continuity of eigenvalues, after shrinking the neighborhood of $x=1$, we may assume that for every $x$ in some punctured neighborhood of $1$, all eigenvalues of $P(x)$ except $\lambda(x)$ also differ from $1$.

To avoid using an inverse operator on a subspace depending on $x$,
introduce the fixed-size matrix
\[
B(x):=I-P(x)+\Pi(x).
\]
At $x=1$, the matrix $B(1)$ is invertible: on $\operatorname{Im}\Pi(1)$, it acts
as multiplication by $2-\lambda(1)=1$, whereas on $\operatorname{Im}Q(1)$ it acts
as $I-P(1)$; here $1$ is a simple eigenvalue of $P(1)$.
Thus, by continuity, $B(x)$ is invertible in some neighborhood of
$x=1$, and $B(x)^{-1}$ is analytic there.

Set
\[
R(x):=Q(x)B(x)^{-1}Q(x).
\]
This is a fixed-size matrix, so its analyticity in a neighborhood of
$x=1$ is immediate.

Consider
\[
S(x):=\frac{\Pi(x)}{1-\lambda(x)}+R(x).
\]
By \eqref{eq:comm-rel},
\[
(I-P)\Pi=\Pi(I-P)=(1-\lambda)\Pi,
\]
and hence,
\[
(I-P)\frac{\Pi}{1-\lambda}=\frac{\Pi}{1-\lambda}(I-P)=\Pi.
\]
Next, since $P,\Pi,Q$ commute pairwise, the matrix $B=I-P+\Pi$ also
commutes with $Q$, and
\[
QB=BQ=Q(I-P).
\]
Therefore,
\[
(I-P)R=(I-P)QB^{-1}Q=QB B^{-1}Q=Q,
\]
and similarly,
\[
R(I-P)=Q.
\]
It follows that
\[
(I-P)S=S(I-P)=\Pi+Q=I.
\]
Thus,
\[
(I-P(x))^{-1}=\frac{\Pi(x)}{1-\lambda(x)}+R(x),
\]
where $R(x)$ is analytic near $x=1$.
\end{proof}
\paragraph{Proof of Theorem \ref{th3}.}
\begin{proof}[]
We determine $\nu_l$ from the limit of $(1-x)H_{0l}(x)$ as $x\uparrow1$. By
Proposition~2, $h^n_{0l}\to\nu_l\mathbf1$; hence Abel's theorem gives
\[
\lim_{x\uparrow1}(1-x)H_{0l}(x)=\nu_l\mathbf1.
\]
To compute this limit, we use the Laurent expansion of
$H_{0l}(x)$ near $x=1$. Recall that $F(x)=xI-xP(x)$ for
$x\neq0$, and at $x=0$ it extends as the polynomial \eqref{xi-px}.
By Lemma~9,
\begin{equation}\label{eq:F-inv}
F^{-1}(x)=x^{-1}\left(\frac{\Pi(x)}{1-\lambda(x)}+R(x)\right),
\end{equation}
where $R(x)$ is analytic in some neighborhood of $1$.
We now seek a Laurent expansion of $H_{0l}(x)$ near $x=1$. Lemma~7 and \eqref{lambda-prime-drift} show that if $v=0$, then $\lambda'(1)=0$ while $\lambda''(1)\neq0$. We therefore have the expansion
\begin{equation}\label{eq:lambda}
1-\lambda(x)=-\frac{\lambda''(1)}{2}(x-1)^2+O\bigl((x-1)^3\bigr),
\qquad x\to 1.
\end{equation}

Let
\[
v_l(x):=xP_1e_l-P_{-1}h^1_{0l}.
\]
Then
\begin{equation}\label{eq:w}
v_l(x)=v_{l,0}+(x-1)P_1e_l,
\qquad
v_{l,0}:=P_1e_l-P_{-1}h^1_{0l}.
\end{equation}

Substituting \eqref{eq:w} into the formula for $H_{0l}(x)$ gives

\begin{equation}\label{eq:H-spectral}
H_{0l}(x)=F^{-1}(x)v_l(x)=x^{-1}\frac{\Pi(x)v_l(x)}{1-\lambda(x)}+x^{-1}R(x)v_l(x).
\end{equation}
The second term is analytic near $x=1$, so the asymptotic behavior is determined solely by the first term.

Introduce the scalar function
\[
\varphi_l(x):=\ell(x)^Tv_l(x).
\]
Since $\Pi(x)=r(x) \circ \ell(x)$, we have
\[
\Pi(x)v_l(x)=r(x)\varphi_l(x).
\]
Consequently, \eqref{eq:H-spectral} yields
\begin{equation}\label{eq:H-varphi}
H_{0l}(x)=x^{-1}\,r(x)\,\frac{\varphi_l(x)}{1-\lambda(x)}+C_2(x),
\end{equation}
where $C_2(x)$ is analytic in some neighborhood of $1$.
We show that
\begin{equation}\label{eq:cancel}
\varphi_l(1)=\pi^T v_{l,0}=\pi^T\bigl(P_1e_l-P_{-1}h^1_{0l}\bigr)=0.
\end{equation}

Indeed, if $\varphi_l(1)\neq0$, then \eqref{eq:lambda} and \eqref{eq:H-varphi}, along the real axis as $x\uparrow1$, would imply
\[
H_{0l}(x)= -\frac{2\,\varphi_l(1)}{\lambda''(1)}\,\frac{\mathbf1}{(x-1)^2}+O\!\left(\frac{1}{x-1}\right).
\]
On the other hand, all components of $h^i_{0l}$ are probabilities and lie in $[0,1]$. Thus, for $0<x<1$, componentwise,
\[
0\le H_{0l}(x)=\sum_{i\ge1}x^{i-1}h^i_{0l}\le\frac{\mathbf1}{1-x},
\]
which rules out growth of order $(1-x)^{-2}$. This contradiction proves \eqref{eq:cancel}. Expanding $\varphi_l(x)$ in a Taylor series near $x=1$, we obtain

\[
\varphi_l(x)=\varphi_l'(1)(x-1)+O\bigl((x-1)^2\bigr),
\]
where
\[
\varphi_l'(1)=\ell'(1)^Tv_{l,0}+\pi^T P_1e_l.
\]
Indeed, since
\[
\varphi_l(x)=\ell(x)^Tv_l(x),
\qquad
v_l(x)=xP_1e_l-P_{-1}h^1_{0l},
\]
we have
\[
\varphi_l'(x)=\ell'(x)^Tv_l(x)+\ell(x)^Tv_l'(x),
\qquad
v_l'(x)=P_1e_l.
\]
Thus, at $x=1$,
\begin{equation}\label{eq:phi-prime}
\varphi_l'(1)=\ell'(1)^Tv_{l,0}+\pi^T P_1e_l
=\ell'(1)^T\bigl(P_1e_l-P_{-1}h^1_{0l}\bigr)+\pi^T P_1e_l.
\end{equation}

Substituting the expansions for $r(x)$, $\varphi_l(x)$, and $1-\lambda(x)$ into \eqref{eq:H-varphi}, we obtain
\[
H_{0l}(x)
= x^{-1}\bigl(\mathbf{1}+O(x-1)\bigr)
\frac{\varphi_l'(1)(x-1)+O((x-1)^2)}{-\frac{\lambda''(1)}{2}(x-1)^2+O((x-1)^3)}
+O(1).
\]
Therefore,
\[
H_{0l}(x)
= -\frac{2\varphi_l'(1)}{\lambda''(1)}\,\frac{\mathbf{1}}{x-1}+O(1)
= \frac{2\varphi_l'(1)}{\lambda''(1)}\,\frac{\mathbf{1}}{1-x}+O(1),
\qquad x\to 1.
\]
Multiplying the last equality by $(1-x)$, taking the limit as $x\uparrow1$, and using Proposition~2 and Abel's theorem, we obtain
\[
\lim_{x\uparrow1}(1-x)H_{0l}(x)=\nu_l\mathbf1.
\]
Consequently,
\begin{equation}\label{eq:nu-via-phi}
\nu_l=\frac{2\varphi_l'(1)}{\lambda''(1)}.
\end{equation}

Substituting \eqref{eq:phi-prime} and \eqref{eq:w}, we finally obtain
\begin{equation} \nonumber
\nu_l
=\frac{2}{\lambda''(1)}
\Bigl[
\pi^T P_1e_l
+\ell'(1)^T\bigl(P_1e_l-P_{-1}h^1_{0l}\bigr)
\Bigr],
\end{equation}
which coincides with (\ref{theor3}) when $a=\lambda''(1)$ and $b^T=\ell'(1)^T$.

We now derive an explicit expression for $\ell'(1)^T$.
From the identity
\[
\ell(x)^TP(x)=\lambda(x)\ell(x)^T
\]
differentiation with respect to $x$ gives
\[
\ell'(x)^TP(x)+\ell(x)^TP'(x)=\lambda'(x)\ell(x)^T+\lambda(x)\ell'(x)^T.
\]
Setting $x=1$,
we obtain
\begin{equation}\label{eq:lprime-basic}
\pi^T P'(1)+\ell'(1)^TP=\ell'(1)^T,
\end{equation}

The left eigenvector can be chosen so that, in some neighborhood of $1$,
\[
\ell(x)^T\mathbf{1}\equiv1.
\]
We simultaneously rescale the right eigenvector by the reciprocal factor, so that the condition $\ell(x)^Tr(x)=1$ remains valid.
Then
\begin{equation}\label{eq:lprime-1}
\ell'(1)^T\mathbf{1}=0.
\end{equation}

In \eqref{eq:lprime-basic}, move all terms involving $\ell'(1)^T$ to the right and the remaining terms to the left, and multiply the equation on the right by $Z=(I-P+\mathbf{1}\circ \pi)^{-1}-\mathbf{1}\circ \pi$:
\[
\ell'(1)^T(I-P)Z=\pi^T P'(1)Z.
\]
By Lemma~8,
\[
(I-P)Z=Z(I-P)=I-\mathbf{1}\circ\pi.
\]
Consequently,
\[
\ell'(1)^T(I-\mathbf{1}\circ \pi)=\pi^T P'(1)Z.
\]
But by \eqref{eq:lprime-1},
\[
\ell'(1)^T(I-\mathbf{1}\circ \pi)=\ell'(1)^T-(\ell'(1)^T\mathbf{1})\pi^T=\ell'(1)^T.
\]
Thus,
\[
b^T=\ell'(1)^T=\pi^T P'(1)Z=\pi^T(P_1-P_{-1})Z.
\]

Now compute $a=\lambda''(1)$. Choose another right eigenvector $\rho(x)$, which differs from the previous one only by normalization.
Differentiate twice the identity
\[
P(x)\rho(x)=\lambda(x)\rho(x).
\]
We obtain
\[
P''(x)\rho(x)+2P'(x)\rho'(x)+P(x)\rho''(x)
=
\lambda''(x)\rho(x)+2\lambda'(x)\rho'(x)+\lambda(x)\rho''(x).
\]
Setting $x=1$ and multiplying on the left by $\pi^T$, taking into account
\[
\rho(1)=\mathbf{1},
\qquad
\pi^T P=\pi^T,
\qquad
\lambda(1)=1,
\qquad
\lambda'(1)=0,
\]
we obtain
\begin{equation}\label{eq:lambda-second-pre}
\lambda''(1)=\pi^T P''(1)\mathbf{1}+2\pi^T P'(1)\rho'(1).
\end{equation}

To compute $\rho'(1)$, differentiate the identity $P(x)\rho(x)=\lambda(x)\rho(x)$ once:
\[
P'(1)\mathbf{1}+P\rho'(1)=\lambda'(1)\mathbf{1}+\rho'(1).
\]
Since $\lambda'(1)=0$, we have
\[
(I-P)\rho'(1)=P'(1)\mathbf{1}.
\]
Impose the normalization $\pi^T\rho(x)\equiv1$, which implies
\[
\pi^T\rho'(1)=0.
\]
Thus the unique solution satisfying $\pi^T\rho'(1)=0$ is
\[
\rho'(1)=ZP'(1)\mathbf{1}.
\]
Substituting this into \eqref{eq:lambda-second-pre}, we obtain
\begin{equation}\label{eq:lambda-second-general}
\lambda''(1)=\pi^T P''(1)\mathbf{1}+2\pi^T P'(1)ZP'(1)\mathbf{1}.
\end{equation}

Now use the explicit form of $P(x)$:
\[
P(x)=xP_1+P_0+x^{-1}P_{-1}.
\]
Then
\[
P''(x)=2x^{-3}P_{-1},
\]
and hence,
\[
P''(1)=2P_{-1}.
\]
From \eqref{eq:lambda-second-general}, we finally obtain
\begin{equation}\label{eq:lambda-second-final}
a=\lambda''(1)=2\pi^T P_{-1}\mathbf{1}+2\pi^T(P_1-P_{-1})Z(P_1-P_{-1})\mathbf{1}.
\end{equation}

The theorem is proved in full.
\end{proof}
\section{Application of the results to a system of independent particles}\label{s4}
Consider a system of independent particles, each moving according to the same
probabilistic laws as the original Markov chain. For every starting
site $(i_0,j_0)$, where $i_0\ge1$ and $j_0\in S$, we start an independent copy
of the chain, denoted by $\xi^{(i_0,j_0)}$. Absorption occurs at the
first hitting time of level $0$.

Let $N_n(j_0,j)$ be the number of particles starting at the sites
$(i,j_0)$, $1\le i\le n$, that are absorbed at $(0,j)$, and let $N_n(j)$ be
the total number of particles whose initial first coordinate does not exceed $n$ and that are absorbed
at $(0,j)$. More formally,
\[
N_n(j_0,j)=\sum_{i=1}^n
I\!\left\{\left(\xi^{(i,j_0)}_{\tau_i^{(i,j_0)}}\right)^2=j\right\},
\]
\[
N_n(j)=\sum_{j_0=1}^N N_n(j_0,j).
\]

\begin{theorem}\label{th4}

Suppose that $v<0$. Then for every $j\in S$,
\[
\frac{N_n(j)}{\sum_{k=1}^N N_n(k)}\xrightarrow[n\to\infty]{a.s.}\nu_j,
\]
where $\nu_j$ is given by \eqref{theor2}.
\end{theorem}

The proof requires the following lemma.

\begin{lemma} Suppose that $v<0$. Then, for any $j_0,j\in S$,
\[
\frac{N_n(j_0,j)}{\mathbf E N_n(j_0,j)}\xrightarrow[n\to\infty]{a.s.}1.
\]
\end{lemma}
\begin{proof}
Fix $j_0,j\in S$ and set
\[
A_i=\left\{\left(\xi^{(i,j_0)}_{\tau_i^{(i,j_0)}}\right)^2=j\right\},
\qquad X_i=I_{A_i}.
\]
By independence of the particles, the random variables $X_i$ are independent. Moreover,
\[
P(A_i)=h_{0j}^{ij_0}.
\]
Set
\[
b_n=\sum_{i=1}^nP(A_i)=\mathbf E N_n(j_0,j).
\]
By Proposition~2, $P(A_i)\to\nu_j$. Since the embedded chain is positive,
$\nu_j>0$, so $b_n\to\infty$ and
\begin{equation}\label{bn-cesaro}
\frac{b_n}{n}\longrightarrow\nu_j.
\end{equation}

For Kolmogorov's strong law of large numbers, it is enough
to verify the condition
\begin{equation}\label{kolmogorov-condition}
\sum_{n=1}^\infty\frac{\operatorname{Var}X_n}{b_n^2}<\infty.
\end{equation}
We have
\[
\sum_{n=1}^\infty\frac{\operatorname{Var}X_n}{b_n^2}
=\sum_{n=1}^\infty\frac{P(A_n)(1-P(A_n))}{b_n^2}
\le\sum_{n=1}^\infty\frac{b_n-b_{n-1}}{b_n^2},
\qquad b_0=0.
\]
The first term of the last series is $1/b_1$. For $n\ge2$, monotonicity of $b_n$
gives
\[
\frac{b_n-b_{n-1}}{b_n^2}
\le \int_{b_{n-1}}^{b_n}\frac{dx}{x^2}.
\]
Consequently,
\[
\sum_{n=1}^\infty\frac{b_n-b_{n-1}}{b_n^2}
\le\frac1{b_1}+\int_{b_1}^{\infty}\frac{dx}{x^2}
=\frac2{b_1}<\infty.
\]
Thus \eqref{kolmogorov-condition} holds, and
\[
\frac{\sum_{i=1}^n(X_i-P(A_i))}{b_n}\xrightarrow{a.s.}0.
\]
Since $\sum_{i=1}^nX_i=N_n(j_0,j)$, the lemma follows.
\end{proof}
\begin{proof}[]
By Lemma~10 and \eqref{bn-cesaro},
\[
\frac{N_n(j_0,j)}n
=\frac{N_n(j_0,j)}{\mathbf E N_n(j_0,j)}
\frac{\mathbf E N_n(j_0,j)}n
\xrightarrow{a.s.}\nu_j
\]
for any $j_0,j\in S$. Hence, by the definition of $N_n(j)$,
\[
\frac{N_n(j)}n
=\sum_{j_0=1}^N\frac{N_n(j_0,j)}n
\xrightarrow{a.s.}N\nu_j.
\]
This holds for every $j\in S$, and therefore,
\[
\frac1n\sum_{k=1}^N N_n(k)
\xrightarrow{a.s.}N\sum_{k=1}^N\nu_k=N.
\]
It follows that
\[
\frac{N_n(j)}{\sum_{k=1}^N N_n(k)}
=\frac{N_n(j)/n}{\sum_{k=1}^N N_n(k)/n}
\xrightarrow{a.s.}\nu_j.
\]
The theorem is proved.
\end{proof}
\section*{Example: a homogeneous strip of height 2}

Consider a random walk in a strip of height $2$. To ensure that this example satisfies
the original assumption of strict positivity of all admissible transition probabilities,
set
\[
B_\theta=
\begin{pmatrix}
1-\theta&\theta\\
\theta&1-\theta
\end{pmatrix},
\qquad 0<\theta<\frac12,
\]
and put
\[
P_1=p_-B_\theta,\qquad
P_0=p_0B_\theta,\qquad
P_{-1}=p_+B_\theta,
\]
where
\[
p_+,p_-,p_0>0,\qquad p_++p_-+p_0=1,\qquad p_+<p_-.
\]
All entries of the three matrices are positive.

The induced chain has transition matrix
\[
P=P_{-1}+P_0+P_1=B_\theta,
\]
hence,
\[
\pi=\binom{1/2}{1/2},\qquad
\pi^T=\left(\frac12,\frac12\right).
\]
Since $B_\theta\mathbf1=\mathbf1$,
\[
m=(P_{-1}-P_1)\mathbf1=(p_+-p_-)\mathbf1,
\]
and therefore,
\[
v=\pi^Tm=p_+-p_-<0.
\]
Thus the assumption of Theorem~\ref{th2} is satisfied.

Let $H$ be the matrix whose $l$-th column is $h^1_{0l}$.
We compute it using the proposed algorithm. We have
\[
G(x)=P_{-1}x^2+(P_0-I)x+P_1
=p_+x^2B_\theta+(p_0B_\theta-I)x+p_-B_\theta.
\]
The matrix $B_\theta$ has two eigenvectors,
\[
u^{(1)}=\binom11,
\qquad
u^{(2)}=\binom1{-1},
\]
with the corresponding eigenvalues
\[
1,\qquad c:=1-2\theta\in(0,1).
\]
Thus $G(x)$ acts independently along these two directions.

On $u^{(1)}$, we obtain
\[
G(x)u^{(1)}
=\bigl(p_+x^2+(p_0-1)x+p_-\bigr)u^{(1)}.
\]
Using $p_++p_-+p_0=1$,
\[
p_+x^2+(p_0-1)x+p_-
=(x-1)(p_+x-p_-).
\]
Thus the roots of the first factor are
\[
x=1,
\qquad
x=\frac{p_-}{p_+}>1,
\]
and only $x=1$ remains in the closed unit disk.

On $u^{(2)}$, we have
\[
G(x)u^{(2)}
=\bigl(cp_+x^2+(cp_0-1)x+cp_-\bigr)u^{(2)}.
\]
The roots of this quadratic polynomial are
\[
x_{\pm}
=\frac{1-cp_0\pm\sqrt{(1-cp_0)^2-4c^2p_+p_-}}{2cp_+}.
\]
By Lemma~4, exactly one of them lies in the open unit disk. Denote it by
\[
\lambda=
\frac{1-cp_0-\sqrt{(1-cp_0)^2-4c^2p_+p_-}}{2cp_+}.
\]
Thus the eigenvalues of $H$ are
\[
\lambda_1=1,\qquad \lambda_2=\lambda,
\]
and the corresponding eigenvectors are $u^{(1)}$ and $u^{(2)}$. Both eigenvalues
are simple, so no generalized eigenvectors arise in this example.

Construct the matrix of the Jordan basis:
\[
S=
\begin{pmatrix}
1&1\\
1&-1
\end{pmatrix},
\qquad
S^{-1}=\frac12
\begin{pmatrix}
1&1\\
1&-1
\end{pmatrix}.
\]
Then
\[
H=S
\begin{pmatrix}
1&0\\
0&\lambda
\end{pmatrix}
S^{-1}
=\frac12
\begin{pmatrix}
1+\lambda&1-\lambda\\
1-\lambda&1+\lambda
\end{pmatrix}.
\]
Consequently,
\[
h^1_{01}
=\binom{(1+\lambda)/2}{(1-\lambda)/2},
\qquad
h^1_{02}
=\binom{(1-\lambda)/2}{(1+\lambda)/2}.
\]

Now compute $\nu$ using Theorem~\ref{th2}. Since
$\pi^TB_\theta=\pi^T$, for $l=1,2$,
\[
\pi^TP_1e_l=\frac{p_-}{2}.
\]
Moreover, $H$ is stochastic and symmetric, hence
$\pi^Th^1_{0l}=1/2$, and
\[
\pi^TP_{-1}h^1_{0l}
=p_+\pi^TB_\theta h^1_{0l}
=p_+\pi^Th^1_{0l}
=\frac{p_+}{2}.
\]
Thus,
\[
\pi^T(P_1e_l-P_{-1}h^1_{0l})
=\frac{p_--p_+}{2}
=-\frac v2.
\]
By Theorem~\ref{th2},
\[
\nu_l=-\frac{-v/2}{v}=\frac12,
\qquad l=1,2.
\]
Therefore,
\[
\nu=\binom{1/2}{1/2}.
\]

\end{document}